\documentclass{amsart}

\newtheorem{theorem}{Theorem}[section]
\newtheorem{lemma}[theorem]{Lemma}

\theoremstyle{definition}
\newtheorem{definition}[theorem]{Definition}

\newtheorem{xca}[theorem]{Exercise}
\newtheorem{proposition}[theorem]{Proposition}

\theoremstyle{remark}
\newtheorem{remark}[theorem]{Remark}

\numberwithin{equation}{section}

\begin{document}

\title[Locally Bronze Semi-Riemannian manifold with an $(l,m)$-type Connection]{Screen Generic Lightlike Submanifolds of a Locally Bronze Semi-Riemannian Manifold equipped with an $(l,m)$-type Connection}

\author{Rajinder Kaur}
\address{Department of Mathematics, Punjabi University,
	Patiala, Punjab, India}
\curraddr{Department of Mathematics, Punjabi University,
	Patiala, Punjab, India}
\email{rajinderjasar@gmail.com}
\author{Jasleen Kaur}
\address{Department of Mathematics, Punjabi University,
	Patiala, Punjab, India}
\email{jasleen2381@gmail.com}



\begin{abstract}
The present paper introduces the geometry of screen generic lightlike submanifolds  of a locally bronze semi-Riemannian manifolds endowed with an (l,m)-type connection. The characterization theorems on geodesicity of such submanifolds with respect to the integrability and parallelism of the distributions are provided. It is proved that there exists no coisotropic , isotropic or totally  proper screen generic lightlike submanifold of a locally bronze semi-Riemannian manifold. Assertions for the smooth transversal vector fields in totally umbilical proper screen generic lightlike submanifold are obtained. The structure of a minimal screen generic lightlike submanifold of a locally bronze semi-Riemannian manifold is detailed with an example.
\end{abstract}

\maketitle

\section{Introduction}
De Spinadel \cite{V1} developed the theory of metallic means family, in which bronze ratio has an important location in studying topics such as dynamical system and quasicrystals. The geometry of bronze structures has an important relation with pure Riemannian metrics according to the corresponding structure. Pandey \cite{B2} introduced the concept of bronze structure on Riemannian manifold which was further studied by Aknipar \cite{B3}. Moreover, Duggal and Bejancu \cite{B1} initiated the theory of lightlike submanifolds of semi-Riemannian manifolds which provides an outstanding framework for geometric characteristics and has versatile applications, particularly in general relativity.\\ 
In view of this, Jin \cite{G6} introduced the idea of generic lightlike submanifolds for an indefinite cosymplectic manifold  which was further explored by \cite{G5,G6} for indefinite Sasakian manifold and indefinite Kaehler manifold, respectively. Several new connections such that semi-symmetric non-metric connection \cite{G2}, non-metric $\phi$- symmetric connection \cite{G4} etc. applied to the screen generic lightlike submanifolds came into existence. As the generalization of screen Cauchy Riemann(SCR) and generic lightlike submanifolds, Dogan et al. \cite{SG1} introduced the notion of screen generic lightlike submanifold for indefinite Kaehler manifold. Then, a few more similar classes namely, contact totally umbilical screen generic and minimal screen generic lightlike submanifolds of indefinite Sasakian manifolds were researched by Gupta \cite{SG5}. Recently, the theory of screen generic lightlike submanifolds was investigated for golden semi-Riemannian manifold \cite{SG2}, cosympletic manifold \cite{SG4} and semi-Riemannian product manifold \cite{SG3}.\\    
Jin \cite{E4} introduced the notion of non-symmetric and non-metric $(l,m)$-type connection on semi-Riemannian manifolds as follows:\\
A linear connection $\bar{\Omega}$ on semi-Riemannian manifold $(\bar{M},\bar{g})$ is called an $(l,m)$-type connection if its torsion tensor $\bar{T}$ satisfy
\begin{equation*}
	\bar{T}(X,Y)=l\{\theta(Y)X-\theta(X)Y\}+m\{\theta(Y)\bar{J}X-\theta(X)\bar{J}Y\}
\end{equation*}
where $l$ and $m$ are smooth functions, $\bar{J}$ is a tensor field of $(1,1)$-type and $\theta$ is a 1-form associated with a smooth unit spacelike vector field $\eta$, which is called the characteristic vector field, by $\theta(X)=\bar{g}(X,\eta)$. We set\;$(l,m)\neq(0,0)$ and denote by $X,Y$ and $Z$, the smooth vector fields on $\bar{M}$. Also, Jin et al. \cite{G3} have developed the geometry of generic lightlike submanifold for indefinite an Kaehler manifold equipped with an $(l,m)$-type connection. However, the concept of $(l,m)$-type connection for a locally bronze semi-Riemannian manifold having screen generic lightlike submanifolds is yet to be investigated.\\
In this paper, we introduce the bronze structure for semi-Riemannian manifold and analyze the geometry of screen generic lightlike submanifolds of a locally bronze semi-Riemannian manifolds equipped with an $(l,m)$-type connection. The integrability, parallelism and geodesicity of the distributions have been characterized. We prove that there exists no coisotropic , isotropic or totally  proper screen generic lightlike submanifold $M$ of a locally bronze semi-Riemannian manifold. Totally umbilical proper screen generic lightlike submanifolds for locally bronze semi-Riemannian manifolds are also worked upon. Furthermore, the structure of a minimal screen generic lightlike submanifolds has been substantiated with an example.
\section{Preliminaries}
Consider ($ \bar{M},\bar g$) as an $(m+n)$-dimensional semi-Riemannian manifold with semi-Riemannian metric $\bar{g}$ of constant index $q$ such that $m,n\geq 1$, $1\leq q\leq m+n-1$.\\
Let $(M,g)$ be a $m$-dimensional lightlike submanifold of $\bar{M}$. In this case, there exists a smooth distribution $RadTM$ on $M$ of rank $r>0$, known as radical distribution on $M$ such that $Rad TM_p = TM_p \cap TM_p^{\perp}, \forall ~p \in M$ where $TM_p$ and  $TM_p^{\perp}$ are degenerate orthogonal spaces but not complementary. Then $M$ is called an $r$-lightlike submanifold of $\bar{M}$. \\
Now, consider $S(TM)$, known as screen distribution, as  a complementary distribution of radical distribution in  $TM$  i.e.,
\[TM = Rad TM  \perp S(TM),\]
and  $S(TM^{\perp})$, called screen transversal vector bundle, as a complementary vector subbundle to $Rad(TM)$ in $TM^{\perp}$ i.e.,
\[
TM^{\perp} = RadTM  \perp S(TM^{\perp}),
\]
As $S(TM)$ is non degenerate vector subbundle of $T\bar{M}{\mid}_M$, we have
\[
T\bar{M}{\mid}_M = S(TM) \perp S(TM)^{\perp},
\]
where $S(TM)^{\perp}$  is the complementary orthogonal vector subbundle of $S(TM)$ in $T\bar{M}{\mid}_M$.\\
Let $tr(TM)$ and $ltr(TM)$ be complementary vector bundles to $TM$ in 
$T\bar{M}{\mid}_M$ and to $RadTM$ in $S(TM^{\perp})^{\perp}$ . Then we have
\[
tr(TM) = ltr(TM) \perp S(TM^{\perp}),
\]
\[
T\bar{M}{\mid}_M = TM \oplus tr(TM), 
\]
\[ = (RadTM \oplus ltr(TM)) \perp S(TM) \perp S(TM^{\perp}).
\]
There exist four cases for a lightlike submanifold $(M,g, S(TM),S(TM^{\perp}))$:\\
(1) $M$ is $r$-lightlike if $r < min(m,n)$,\\
(2) $M$ is co-isotropic if $r=n < m$, i.e. $S(TM^{\perp})=\{0\}$,\\
(3) $M$ is isotropic if $r=m < n $, i.e. $S(TM)=\{0\}$ and\\
(4) $M$ is totally lightlike if $r= n = m$, i.e. $S(TM^{\perp})=\{0\}= S(TM)$.
\begin{theorem}
	\cite{B1} Let $(M,g,S(TM), S(TM^{\perp}))$ be an $r$-lightlike submanifold of a semi-Riemannian manifold $(\bar{M},\bar{g})$. Then there exists a complementary vector bundle $ltr(TM)$ called a lightlike transversal bundle of $RadTM$ in $S(TM^{\perp})^{\perp}$ and basis of $\Gamma(ltr(TM){\mid}_U)$ consisting of smooth sections $\{N_1,\cdots,N_r\}$ $S(TM^{\perp})^{\perp}{\mid}_U$, where $U$ is a coordinate neighborhood of $M$ such that 
	\[
	\bar{g}(N_i,\xi_j) = \delta_{ij} , \quad \bar{g}(N_i,N_j) = 0, \quad i,j=0,1,\cdots , r
	\]
	where  $\{{\xi_1, \cdots , \xi_r}\}$ is a lightlike basis of $\Gamma(RadTM){\mid}_U$.
\end{theorem}
Let $\bar{\nabla}$ be the Levi-Civita connection on $\bar{M}$. Then, the corresponding  Gauss and Weingarten formulae are follows as:
\begin{equation}\label{eq21}
	\bar\nabla_X Y = \nabla_X Y +h(X,Y), \quad \forall ~ X, Y\in \Gamma(TM),
\end{equation}
\begin{equation}
	\bar\nabla_X N = -A_N X +  \nabla_X^{t}N, \quad \forall ~ X \in \Gamma(TM),N \in \Gamma(tr(TM)),
\end{equation}
where $\{\nabla_X Y,-A_N X\}$ and $\{h(X,Y),\nabla_X^{t}N\}$ belong to $\Gamma(TM)$ and $\Gamma(tr(TM))$, respectively. $\nabla$ and $\nabla^{t}$ are linear connections on $M$ and on the vector bundle $tr(TM)$, respectively.\\
Considering the projection morphisms $L$ and $S$ of $tr(TM)$ on $ltr(TM)$ and $S(TM^{\perp})$ respectively, we have
\begin{equation}\label{1}
	\bar\nabla_X Y =  \nabla_X Y +h^{l}(X,Y)+h^{s}(X,Y),
\end{equation}
\begin{equation}\label{2}
	\bar\nabla_X N = -A_N X +  \nabla_X^{l}N+D^{s}(X,N),
\end{equation}
\begin{equation}\label{3}
	\bar\nabla_X W = -A_W X + \nabla_X^{s}W+D^{l}(X,W),
\end{equation}
where\;$h^{l}(X,Y)=Lh(X,Y) ,h^{s}(X,Y)=Sh(X,Y), \{\nabla_X Y,A_N X,A_W X \}\in\Gamma(TM)$, $\{\nabla_X^{l}N,D^{l}(X,W)\}\in\Gamma(ltr(TM))$ and $\{\nabla_X^{s}W,D^{s}(X,N)\}\in\Gamma(S(TM^{\perp}))$. Then, considering $(\ref{1})-(\ref{3})$ and the fact that $\bar{\nabla}$ is a metric connection,  the following holds: 
\begin{equation}\label{4}
	\bar{g}(h^{s}(X,Y),W)+\bar{g}(Y,D^{l}(X,W))=\bar{g}(A_W X,Y),
\end{equation}
\begin{equation}\label{5}
	\bar{g}(D^{s}(X,N),W)=\bar{g}(A_W X,N).
\end{equation}
Let $J$ be a projection of $TM$ on $S(TM)$. Then, we have
\begin{equation}\label{6}
	\nabla_X JY= \nabla_X^{*}JY + h^{*}(X,JY),
\end{equation}
\begin{equation}\label{7}
	\nabla_X \xi=-A_\xi^{*}X+\nabla_X^{*t}\xi,
\end{equation}
for any $X,Y\in\Gamma(TM)$ and $\xi\in\Gamma(Rad(TM))$, where$\{\nabla_X^{*}JY,A_\xi^{*}X\} $ and $\{(h^{*}(X,JY),\nabla_X^{*t}\xi\}$ belong to $\Gamma(S(TM))$ and $\Gamma(RadTM)$, respectively.\\
By using the above equations, we obtain 
\begin{equation}\label{8}
	\bar{g}(h^{l}(X,JY),\xi)=g(A_\xi^{*}X,JY),
\end{equation}
\begin{equation}\label{9}
	\bar{g}(h^{*}(X,JY),N)=g(A_N X,JY),
\end{equation}
\begin{equation}\label{10}
	\bar{g}(h^{l}(X,\xi),\xi)=0, A_\xi^{*}\xi=0.
\end{equation}
Generally, $\nabla$ on $M$ is not metric connection. Since $\bar{\nabla}$ is a metric connection, from $(\ref{1})$, we derive
\begin{equation}\label{11}
	(\nabla_X g)(Y,Z)=\bar{g}(h^{l}(X,Y),Z)+\bar{g}(h^{l}(X,Z),Y),
\end{equation}\\
for any $X,Y,Z\in\Gamma(TM)$. Here, $\nabla^{*}$ is a metric connection on $S(TM)$.
\begin{definition}
	A polynomial structure on a semi-Riemannian manifold $\bar{M}$ is known as bronze structure if it is determined by $\bar{J}$ such that
	\begin{equation}\label{a}
		\bar{J}^{2}=3\bar{J}+I,
	\end{equation}
	If a semi-Riemannian metric $\bar{g}$ satisfies the equation
	\begin{equation}\label{b}
		\bar{g}(X,\bar{J}Y)=\bar{g}(\bar{J}X,Y),
	\end{equation}
	which yields
	\begin{equation}\label{c}
		\bar{g}(\bar{J}X,\bar{J}Y)=3\bar{g}(X,\bar{J}Y)+\bar{g}(X,Y),\quad \forall \quad X,Y\in\Gamma(T\bar{M})
	\end{equation}
	then $\bar{g}$ is called $\bar{J}$-compatible.
\end{definition}
\begin{definition}
	If semi-Riemannian metric $\bar{g}$ is compatible with the bronze structure  $\bar{J}$, then the pair $(\bar{g},\bar{J})$ is called a bronze semi-Riemannian structure and $(\bar{M},\bar{g},\bar{J})$ is {bronze semi-Riemannian manifold}.
\end{definition}
\begin{definition}
	If $(\bar{M},\bar{g},\bar{J})$ is a bronze semi-Riemannian manifold and 
	\begin{equation}\label{b1}
		\bar \nabla \bar{J}=0,
	\end{equation} 
	then  we say that $(\bar{M},\bar{g},\bar{J})$ is a locally bronze semi- Riemannian manifold.
\end{definition}
Let $(\bar{M},\bar{g},\bar{J})$ be a locally bronze semi-Riemannian manifold, where $\bar{g}$ is a semi-Riemannian metric and $\bar{J}$ is a bronze structure. Then,\\
For each $X$ tangent to $M$, $\bar{J}X$ can be written as follows
\begin{equation}\label{s1}
	\bar{J}X=fX + wX =fX + w_l X + w_s X,
\end{equation}
where $fX$ and $wX$ are the tangential and the transversal parts of $\bar{J}X$, $w_l$ and $w_s$ are the projections on $ltr(TM)$ and $S(TM^{\perp})$, respectively. In addition, for any $V\in\Gamma(tr(TM))$, $\bar{J}V$ can be written as
\begin{equation}\label{s2}
	\bar{J}V = BV + CV, 
\end{equation} 
where $BV$ and $CV$ are the tangential and the transversal parts of $\bar{J}V$, respectively.\\

\section{Screen generic lightlike submanifolds}
In this section, we introduce  screen generic lightlike submanifold for a locally bronze semi-Riemannian manifold.
\begin{definition}
	A real r-lightlike submanifold $M$ of locally bronze semi-Riemannian manifold $\bar{M}$ is said to be a screen generic lightlike submanifold of $\bar{M}$ if the following conditions are satisfied:\\
	(a) $Rad(TM)$ is invariant with respect to $\bar{J}$, that is 
	\begin{equation}\label{g1}
		\bar{J} (RadTM)= RadTM,
	\end{equation}
	(b) There exist a subbundle $B_o$ of $S(TM)$ such that 
	\begin{equation}\label{g2}
		B_o = \bar{J}(S(TM))\cap S(TM),
	\end{equation}
	where $B_o$ is a non-degenerate distribution on $M$.\\
	The above definition implies that there exist a complementary non-degenerate distribution $B^{'}$ in $S(TM)$ such that 
	\begin{equation}\label{g3}
		S(TM)=B_o \oplus B^{'} , \quad \bar{J}(B^{'})\not\subseteq S(TM),\quad \bar{J}(B^{'})\not\subseteq S(TM^{\perp}).
	\end{equation}
	Let $J_o$, $J_1$ and $Q$ be the projection morphisms on $B_o$, $RadTM$ and $B^{'}$ respectively.\\
	Then for all $X\in\Gamma(TM)$,
	\begin{equation}\label{g4}
		X=J_o X + J_1 X+QX = J^{'}X + QX,
	\end{equation}
	where $ B = B_o \perp RadTM$, $B$ is invariant and $ J^{'}X\in\Gamma(B)$, $QX\in\Gamma(B^{'})$.\\
	It is clear that  $\bar{J}(B^{'})\neq B^{'}$.\\
	Now for a vector field $Y\in\Gamma(B^{'})$  and using $(\ref{s1})$, we have 
	\begin{equation}\label{g6}
		\bar{J}Y = fY + wY,
	\end{equation}
	where $fY\in\Gamma(B^{'})$ and $wY\in\Gamma(S(TM^{\perp}))$.\\
	
	$M$ is said to be a proper screen generic lightlike submanifold of a locally bronze semi-Riemannian manifold if $B_o\neq \{0\}$ and $ B^{'}\neq\{0\}$. 
\end{definition}
\begin{proposition}
	Let $M$ be an screen generic lightlike submanifold of a locally bronze semi-Riemannian manifold $(\bar{M},\bar{g},\bar{J})$. Then, the distribution $ltr(TM)$ is invariant with respect to $\bar{J}$.
\end{proposition}

\section{$(l,m)$-type connection}

For Levi-Civita connection $\bar{\nabla}$ on the locally bronze semi-Riemannian manifold$(\bar{M},\bar{g},\bar{J})$, we set
\begin{equation}\label{e1}
	\bar{\Omega}_X Y =\bar{\nabla}_X Y+\theta(Y)\{lX+m\bar{J}X\},
\end{equation}
for any $X,Y\in\Gamma(TM)$. Since $\bar{\nabla}$ is torsion free and metric connection, therefore we obtain 
\begin{equation}\label{e2}
	(\bar{\Omega}_X \bar{g})(Y,Z)=-l\{\theta(Y)\bar{g}(X,Z)+\theta(Z)\bar{g}(X,Y)\}-m\{\theta(Y)\bar{g}(\bar{J}X,Z)+\theta(Z)\bar{g}(\bar{J}X,Y)\},
\end{equation} 
\begin{equation}\label{e3}
	\bar{T}(X,Y)=l\{\theta(Y)X-\theta(X)Y\}+m\{\theta(Y)\bar{J}X-\theta(X)\bar{J}Y\},
\end{equation} 
for any $X,Y,Z\in\Gamma(TM)$, where $\bar{T}^{\bar{\Omega}}$ is a torsion tensor of the connection 
$\bar{\Omega}$, where $l$ and $m$ are smooth functions and $\theta$ is a 1-form associated with a smooth unit spacelike vector field $\eta$, which is called the characteristic vector field, by $\theta(X)=\bar{g}(X,\eta)$. Setting $(l,m)\neq(0,0)$, $\bar{\Omega}$ is called an $(l,m)$-type connection. Since $\bar{M}$ admits a tensor field $\bar{J}$ of type $(1,1)$, therefore for any $X,Y\in\Gamma(TM)$, we have
\begin{equation}\label{e4}
	(\bar{\Omega}_X \bar{J})(Y)=l\{\theta(\bar{J}Y)X-\theta(Y)\bar{J}X\}+m\{\theta(\bar{J}Y)\bar{J}X-3\theta(Y)\bar{J}X-\theta(Y)X\},
\end{equation}
\begin{equation}\label{e5}
	\bar{\Omega}_X \bar{J}Y=\bar{J}(\bar{\Omega}_X Y)+l\{\theta(\bar{J}Y)X-\theta(Y)\bar{J}X\}+m\{\theta(\bar{J}Y)\bar{J}X-3\theta(Y)\bar{J}X-\theta(Y)X\},
\end{equation}

Let $(M,g,S(TM),S(TM^{\perp}))$ be a screen generic lightlike submanifold of a locally bronze semi-Riemannian manifold $(\bar{M},\bar{g})$ with an $(l,m)$ type connection $\bar{\Omega}$. Let $\Omega$ be the induced linear connection on $M$ from $\bar{\Omega}$. Therefore the
Gauss formula is as follows:
\begin{equation}\label{e6}
	\bar{\Omega}_X Y=\Omega_X Y+ \bar{h^{l}}(X,Y)+\bar{h^{s}}(X,Y),
\end{equation}
for any $X,Y\in\Gamma(TM)$, where $\Omega_X Y\in\Gamma(TM)$ , and $\bar{h^{l}}$, $\bar{h^{s}}$
are lightlike second fundamental form and the screen
second fundamental form of $M$, respectively. Now from $(\ref{1})$, $(\ref{s1})$, $(\ref{e1})$ and $(\ref{e6})$, we obtain
\begin{equation}\label{e9}
	\Omega_X Y= {\nabla}_X Y + l\theta(Y)X + m\theta(Y) fX,
\end{equation}
\begin{equation}\label{e10}
	\bar{h^{l}}(X,Y)= h^{l}(X,Y) + m\theta(Y)w_l X,
\end{equation}
\begin{equation}\label{e11}
	\bar{h^{s}}(X,Y)= h^{s}(X,Y) + m\theta(Y)w_s X,
\end{equation}
Moreover using, $(\ref{s1})$, $(\ref{g6})$, $(\ref{e2})$ and $(\ref{e6})$, we get,\\
\begin{equation}\label{e18}
	\begin{split}
		(\Omega_X g)(Y,Z) = g(h^{l}(X,Y),Z) + g(Y,h^{l}(X,Z))- l(\theta(Y)g(X,Z)+\theta(Z)g(Y,X))-\\
		m(\theta(Y)g(fX,Z) + \theta(Z)g(Y,fZ)),
	\end{split}
\end{equation}
\begin{equation}\label{e19}
	T^{\Omega}(X,Y) = l\{\theta(Y)X-\theta(X)Y\}+m\{\theta(Y)fX-\theta(X)fY\},
\end{equation}
for any $X,Y,Z \in\Gamma(TM)$, where $T^{\Omega}$ is torsion tensor of the induced connection $\Omega$ on $M$. Hence, the following result holds:
\begin{theorem}
	Let $M$ be a screen generic lightlike submanifold of a locally bronze semi-Riemannian manifold $\bar{M}$ with an $(l,m)$-type connection $\bar{\Omega}$. Then the induced connection $\Omega$ on the lightlike submanifold $M$ is also an $(l,m)$-type connection.
\end{theorem}
Suppose that $\bar{h^{l}}$ vanishes identically on $M$. Therefore
\begin{equation}\label{e18}
	(\Omega_X g)(Y,Z) = -l(\theta(Y)g(X,Z)+\theta(Z)g(Y,X))-m(\theta(Y)g(fX,Z) +\theta(Z)g(Y,fZ)),
\end{equation}
follows from $(\ref{e18})$.\\
Consequently, we attain the following result:
\begin{theorem}
	Let $M$ be a screen generic lightlike submanifold of a locally bronze semi-Riemannian manifold with an $(l,m)$-type connection $\bar{\Omega}$. Then the induced connection $\Omega$ on the lightlike submanifold $M$ is also an $(l,m)$-type connection if and only if $\bar{h^{l}}$ vanishes identically on $M$ and the characterstic vector field $\zeta\in\Gamma(S(TM^{\perp}))$ such that $\theta(X)=\bar{g}(X,\zeta)$.
\end{theorem}
Corresponding to an $(l,m)$-type connection $\bar{\Omega}$,  the Weingarten formulae are given by: 
\begin{equation}\label{e7}
	\bar{\Omega}_X N = -\bar{A}_N X + \bar{\Omega}_X ^{l}N + \bar{D^{s}}(X,N),
\end{equation}
\begin{equation}\label{e8}
	\bar{\Omega}_X W= -\bar{A}_W X + \bar{\Omega}_X ^{s}W + \bar{D^{l}}(X,W),
\end{equation}
$N\in\Gamma(ltr(TM))$ and $W\in\Gamma(S(TM^{\perp}))$, where $\{\bar{A}_N X, \bar{A}_W X\} \in\Gamma(TM)$, $\{\bar{\Omega}_X ^{l}N,\bar{D^{l}}(X,W)\}\in\Gamma(ltr(TM))$ and $\{\bar{\Omega}_X ^{s}W,\bar{D^{s}}(X,N)\}\in\Gamma(S(TM^{\perp}))$. $\bar{\Omega}^{l}$ and  $\bar{\Omega}^{s}$ are linear connections on $ltr(TM)$ and $S(TM^{\perp})$, respectively. Both $\bar{A}_N$ and $\bar{A}_W$ are linear operators on $\Gamma(TM)$. Now, employing equations $(\ref{2})$, $(\ref{3})$, $(\ref{e1})$, $(\ref{e7})$ and $(\ref{e8})$, we attain
\begin{equation}\label{e12}
	\bar{A}_N X = A_N X - l\theta(N)X - m\theta(N)fX,
\end{equation}
\begin{equation}\label{e13}
	\bar{\Omega}_X^{l}N = \nabla_X ^{l}N + m \theta(N)w_l X,
\end{equation}
\begin{equation}\label{e14}
	\bar{D^{s}}(X,N) = D^{s}(X,N) + m \theta(N) w_s X,
\end{equation}
\begin{equation}\label{e15}
	\bar{A}_W X = A_W X - l\theta(W)X - m\theta(W)fX,
\end{equation}
\begin{equation}\label{e16}
	\bar{\Omega}_X^{s}W = \nabla_X ^{s}W + m \theta(W)w_s X,
\end{equation}
\begin{equation}\label{e17}
	\bar{D^{l}}(X,W) = D^{l}(X,W) + m \theta(W) w_l X,
\end{equation}
Now, using equations $(\ref{4})$, $(\ref{5})$, $(\ref{e11})$, $(\ref{e14})$ and $(\ref{e17})$, we obtain
\begin{equation}\label{e20}
	\begin{split}
		\bar{g}(\bar h^{s}(X,Y),W) + \bar{g}(Y,\bar D^{l}(X,W))= \bar{g}(\bar{A}_W X,Y)+l\theta(W)\bar{g}(X,Y)+m\theta(W)\bar{g}(fX,Y)+\\
		m\theta(Y)\bar{g}(w_s X,W) + m\theta(W)\bar{g}(X,w_l X),
	\end{split}
\end{equation}
\begin{equation}\label{e21}
	\bar{g}(\bar D^{s}(X,N),W)=\bar{g}(\bar{A}_W X,N)+ l\theta(W)\bar{g}(X,N)+m\theta(W)\bar{g}(fX,N)+m\theta(N)\bar{g}(w_s X,W),
\end{equation}
Let $J$ be the projection of $TM$ on $S(TM)$, then any $X\in\Gamma(TM)$ can be written as $X=JX+\Sigma_{i=1}^{r} \eta_{i}(X)\xi_{i}$, 
\begin{equation}\label{z1}
	\eta_{i}(X)=g(X,N_{i}),
\end{equation} where $\{\xi_{i}\}_{i=1}^{r}$ is a basis for $Rad\;TM$. Therefore, we have the following decompositions w.r.t $\Omega$
\begin{equation}\label{e22}
	\Omega_X JY = \Omega_X ^{*} JY + \bar{h^{*}}(X,JY),
\end{equation}
\begin{equation}\label{e23}
	\Omega_X \xi = -\bar{A}_\xi ^{*} X + \bar{{\Omega}}_X ^{*t} \xi,
\end{equation}
for any $X,Y\in\Gamma(TM)$, where $\{\Omega_X ^{*} JY, \bar{A}_\xi ^{*} X\} \in\Gamma(S(TM))$ and $\{\bar{h^{*}}(X,JY), \bar{{\Omega}}_X ^{*t} \xi\} \in\Gamma(RadTM)$.\\
From $(\ref{6})$, $(\ref{7})$, $(\ref{s1})$, $(\ref{e22})$  and  $(\ref{e23})$, we obtain
\begin{equation}\label{e24}
	\Omega_X ^{*} JY = {\nabla}_X ^{*} JY + m \theta(JY)JfX + l \theta(JY)JX,
\end{equation}
\begin{equation}\label{e25}
	\bar{h^{*}}(X,JY)= h^{*}(X,JY) + l\theta(JY)\Sigma_{i=1}^{r} \eta_{i}(X)\xi_{i} + m\theta(JY) \Sigma_{i=1}^{r} \eta_{i}(fX)\xi_{i},
\end{equation} 
\begin{equation}\label{e26}
	\bar{A}_\xi ^{*} X= A_{\xi}^{*}X - l\theta(\xi)JX - m\theta(\xi)JfX,
\end{equation}
\begin{equation}\label{e27}
	\bar{{\Omega}}_X ^{*t} \xi = {{\nabla}}_X ^{*t} \xi + l\theta(\xi)\eta(X)\xi  + m\theta(\xi) \eta(fX)\xi,
\end{equation}
Further, using (\ref{8}), (\ref{9}), (\ref{10}), (\ref{e25}) and (\ref{e26}), we derive
\begin{equation}\label{e28}
	g(\bar{h^{l}}(X,JY),\xi) = g(\bar{A}_\xi ^{*}X, JY) + l\theta(\xi)g(JX,JY) + m\theta(\xi)g(JfX,JY)+m\theta(JY) g( w_l X,\xi),
\end{equation}
\begin{equation}\label{e29}
	\begin{split}
		g(\bar{h^{*}}(X,JY),N)=g (\bar{A}_N X,JY)+l\theta(N)g(X,JY)+ l\theta(JY)\eta(X)+m \theta(N)g(fX,JY)+\\
		m\theta(JY)\eta(fX),
	\end{split}
\end{equation}
\begin{equation}\label{e30}
	g(\bar{h^{l}}(X,\xi),\xi)=m\theta(\xi)g(w_l X,\xi),
\end{equation}
\begin{equation}\label{e31}
	\bar{A}_\xi ^{*} \xi = -l\theta(\xi)J\xi-m\theta(\xi)Jf\xi, 
\end{equation}

\section{Screen generic lightlike submanifolds with an $(l,m)$-type connection}
In the forthcoming part, the characterization results of a screen generic lightlike submanifold for a locally bronze semi-Riemannian manifold are given when the manifold is endowed with an $(l,m)$-type connection.
\begin{definition}
	Let $M$ be a lightlike submanifold of a locally bronze semi-Riemannian manifold $\bar{M}$. If
	\begin{equation}\label{21}
		\bar{J}Rad(TM)=Rad(TM),\quad \bar{J}S(TM)=S(TM),
	\end{equation}
	then $M$  is said to be an invariant lightlike submanifold of a locally bronze semi-Riemannian manifold $\bar{M}$.
\end{definition}
\begin{proposition}
	There exists no coisotropic , isotropic or totally  proper screen generic lightlike submanifold $M$ of a locally bronze semi-Riemannian manifold. If $M$ is screen generic isotropic, coisotropic or totally lightlike submanifold, then it is invariant lightlike submanifold.
\end{proposition}
\begin{proof}
	Let $M$ is a proper screen generic lightlike submanifold of a locally bronze semi-Riemannian manifold $\bar{M}$. Let $M$ is isotropic, $S(TM)=\{0\}$ which implies that $D_o=\{0\}$ and $D^{'}=\{0\}$. Therefore, $\bar{J}Rad(TM)= Rad(TM)$. Hence, $TM= Rad(TM)= \bar{J}Rad(TM)$  which shows that $M$ is invariant submanifold with respect to $\bar{J}$. If $M$ is coisotropic, then $S(TM^{\perp})=\{0\}$. Therefore, using $(\ref{ppp9})$, we obtain $\mu=0$ and $\bar{J}(D^{'})=0$. Also, $TM= D_o\oplus\bar{J}(D^{'})\oplus Rad(TM)$ and $M$ is invariant with respect to $\bar{J}$. Further, if $M$ is totally lightlike, then $S(TM)=\{0\}$ and $S(TM^{\perp})=\{0\}$. Thus, we obtain $TM= Rad(TM)$  which proves $M$ is invariant.\\
	Hence, there exist no coisotropic, isotropic or totally lightlike proper screen generic lightlike submanifolds.       
\end{proof}
\begin{proposition}
	A screen Cauchy Riemann (SCR) lightlike submanifold is a screen generic lightlike submanifold such that distribution $D^{'}$ is totally anti-invariant , that is, 
	\begin{equation}\label{ppp9}
		S(TM^{\perp})= \omega D^{'}\oplus \mu,
	\end{equation}
	where $\mu$ is a non-degenerate invariant distribution.
\end{proposition}
\begin{xca}
	Consider a locally bronze semi-Riemannian manifold $\bar{M}= (R_2^{16},\bar{g})$ of signature $(-,-,+,+,+,+,+,+,+,,+,+,+,+,+,+,+)$ with respect to the basis $\{\partial z_1,\partial z_2,\partial z_3,\partial z_4,\partial z_5,\partial z_6,\partial z_7,\partial z_8,\partial z_9,\partial z_{10},\partial z_{11},\partial z_{12},\partial z_{13},\partial z_{14},\partial z_{15},\partial z_{16}\}$, where the bronze structure $\bar{J}$ is defined by $\bar{J}(z_1,z_2,z_3,z_4,z_5,z_6,z_7,z_8,z_9,z_{10},z_{11},z_{12},z_{13},z_{14},z_{15},z_{16})=(\sigma z_1,\sigma{z_2},\sigma{z_3},\sigma{z_4},(3-\sigma){z_5},\sigma{z_6},\sigma{z_7},\sigma{z_8},(3-\sigma){z_9},\sigma{z_{10}},(3-\sigma){z_{11}},\sigma{z_{12}},3{z_{13}}+{z_{14}},{z_{13}},3{z_{15}}+{z_{16}},{z_{15}})$.\\
	Consider a submanifold $M$ of $\mathbb{R}_{2}^{16}$ is given by the equations 
	\begin{equation*}
		z_1=y_1-y_2,\quad z_2= \sigma y_4,
	\end{equation*}
	\begin{equation*}
		z_3=y_3+y_5,  \quad   z_4=y_2+y_3,
	\end{equation*}
	\begin{equation*}
		z_5=y_4,\quad z_6=y_3-y_5,
	\end{equation*}
	\begin{equation*}
		z_7=y_2+y_1,\quad z_8= -y_2+y_3,\quad z_9=-cos\alpha\;{y_6},
	\end{equation*}
	\begin{equation*}
		z_{10}=-cos\alpha\;{y_7},\quad z_{11}=sin\alpha\;{y_6},\quad z_{12}=sin\alpha\;{y_7},
	\end{equation*}
	\begin{equation*}
		z_{13}=-siny_7\;coshy_8,\quad z_{14}={0},\quad z_{15}=cosy_7 \;sinhy_8,\quad z_{16}=0.
	\end{equation*}
	
	Here $TM$ is spanned by $\{B_1,B_2,B_3,B_4,B_5,B_6,B_7,B_8,B_9\}$, where
	\begin{equation*}
		B_1=-\partial{z_1}+\partial{z_4}+\partial{z_7}-\partial{z_8},
		\quad B_2=\partial{z_3}+\partial{z_4}+\partial{z_6}+\partial{z_8},
	\end{equation*}
	\begin{equation*}
		B_3=\partial{z_1}+\partial{z_7}, \quad B_4=\partial{z_3}-\partial{z_6},\quad B_5={\sigma}\partial{z_2}+\partial{z_5},
	\end{equation*}
	\begin{equation*}
		B_6= -cos\alpha\;{\partial z_9}+ sin\alpha\;{\partial{z_{11}}},\quad B_7= -cos\alpha\;{\partial z_{10}}+ sin\alpha\;{\partial{z_{12}}},
	\end{equation*}
	\begin{equation*}
		B_8= -cosy_7\;coshy_8\;{\partial{z_{13}}}+ siny_7 \; sinhy_8\;{\partial{z_{15}}},\quad
		B_9= cosy_7\;coshy_8\;{\partial{z_{15}}}+ siny_7 \;sinhy_8\;{\partial{z_{13}}},
	\end{equation*}
	Thus, $M$ is a $2$-lightlike submanifold with $Rad(TM)=Span \{B_1,B_2 \}$ such that $\bar{J}B_1=\sigma B_1$ and $\bar{J}(B_2)=\sigma B_2$. Therefore, $Rad(TM)$ is invariant with respect to $\bar{J}$.
	Since $\bar{J}B_3=\sigma B_3$,\; $\bar{J}B_4=\sigma B_4$,\; $\bar{J}B_6=(3-\sigma) B_6$ and $\bar{J}B_7=\sigma B_7$. Therefore, $B_o=Span\{B_3,B_4,B_6,B_7\}$ is also invariant with respect to $\bar{J}$.\\
	Moreover, $S(TM^{\perp})$ is spanned by $\{W_1,W_2,W_3,W_4,W_5\}$, where
	\begin{equation*}
		W_1= sin\alpha\;{\partial{z_9}}+cos\alpha\;{\partial z_{11}},\quad	W_2=-{\partial{z_2}}+\sigma \partial{z_5},\quad W_3= sin\alpha\;{\partial{z_{10}}}+cos\alpha\;{\partial z_{12}}, \\
	\end{equation*}
	\begin{equation*}
		W_4= siny_7 \;sinhy_8\;{\partial{z_{14}}}+ cosy_7\;coshy_8\;{\partial{z_{16}}},\quad
		W_5= -cosy_7\;coshy_8\;{\partial{z_{14}}}+ siny_7 \;sinhy_8\;{\partial{z_{16}}},
	\end{equation*}
	We derive, $\bar{J}B_5=3B_5-W_2$,\; $\bar{J}B_8=3 B_8+ W_5$ and $\bar{J}B_9=3B_9+ W_4$, which means  $B^{'}=Span\{B_5,B_8,B_9\}$, is not invariant with respect to $\bar{J}$ and $\bar{J}(B^{'})\not\subset S(TM^)$ and  $\bar{J}(B^{'})\not\subset S(TM^{\perp})$.\\ 
	Therefore, the lightlike transversal vector bundle $ltr(TM)$ is spanned by
	\begin{equation*}
		N_1=\frac{1}{4}(-\partial{z_1}-\partial{z_4}+\partial{z_7}+\partial{z_8}),\quad N_2=\frac{1}{4}(\partial{z_3}-\partial{z_4}+\partial{z_6}-\partial{z_8}),
	\end{equation*}
	which is invariant with respect to $\bar{J}$.
	Further, $\bar{J}W_1=(3-\sigma) W_1$,\; $\bar{J}W_3=\sigma W_3$, which shows that $\bar{J}\mu=\mu$, where $\mu=Span\{W_1,W_3\}$ i.e., $\mu$ is invariant with respect to $\bar{J}$. \\
	
	Hence, $M$ become a proper screen generic 2-lightlike submanifold of $R_2^{16}$.
\end{xca}
\begin{definition}
	Let $M$ be a screen generic lightlike submanifold of a locally bronze semi-Riemannian manifold $(\bar{M},\bar{J},\bar{g})$ with an $(l,m)$-type connection $\bar{\Omega}$. Then the distribution $B$ is said to be integrable if  and only if $[X,Y]\in\Gamma(B)$ for any $X,Y\in\Gamma(B)$.
\end{definition}
\begin{remark}
	The above definition implies that the distribution $B$ is said to be integrable if and only if $g([X,Y],Z)=0$ for any $X,Y\in\Gamma(B)$ and $Z\in\Gamma(B^{\perp})$.
\end{remark}
\begin{theorem}
	Let $M$ be a screen generic lightlike submanifold of a locally bronze semi-Riemannian manifold $\bar{M}$ with an $(l,m)$- type connection $\bar{\Omega}$. Then, $B_o$ is integrable if and only if the following conditions hold:\\ 
	(i)\;${g}(\bar{h^{*}}(X,\bar{J}Y),\bar{J}N)+3{g}(\bar{h^{*}}(Y,\bar{J}X),N)={g}(\bar{h^{*}}(Y,\bar{J}X),\bar{J}N)+3(\bar{h^{*}}(X,\bar{J}Y),N)$,\\
	(ii)\;${g}(\Omega_X^{*}\bar{J}Y,fZ)+{g}(\bar h^{s}(X,\bar{J}Y),\bar{J}Z)+3{g}(\Omega_Y^{*}\bar{J}X,Z)={g}(\Omega_Y^{*}\bar{J}X,fZ)+{g}(\bar h^{s}(Y,\bar{J}X),\bar{J}Z)+3{g}(\Omega_X^{*}\bar{J}Y,Z),$\\ 
	for any $X,Y \in\Gamma(B_o)$, $Z\in\Gamma(B^{'})$ and $N\in\Gamma(ltr(TM))$.
\end{theorem}
\begin{proof}
	$B_o$ is integrable if and only if\; $\bar{g}([X,Y],N)=0,\quad \bar{g}([X,Y],Z)=0,$\\
	Therefore, from equations $(\ref{e1})$, (\ref{a}) and (\ref{g3}), we obtain\\
	$\bar{g}([X,Y],N)=\bar{g}(\bar{\Omega}_X Y,N)-l\theta(Y)\bar{g}(X,N)-m\theta(Y)\bar{g}(\bar{J}X,N)-\bar{g}(\bar{\Omega}_Y X,N)+l\theta(X)\bar{g}(Y,N)+m\theta(X)\bar{g}(\bar{J}Y,N) = \bar{g}(\bar{\Omega}_X \bar J (\bar{J}Y),N)-3\bar{g}(\bar{\Omega}_X \bar{J}Y,N)-\bar{g}(\bar{\Omega}_Y \bar J (\bar{J}X),N)+3\bar{g}(\bar{\Omega}_Y \bar{J}X,N)$,\\
	Employing equations (\ref{b}), (\ref{e5}), (\ref{e6}) and (\ref{e22}), we have\\
	${g}(\bar{h^{*}}(X,\bar{J}Y)-\bar{h^{*}}(Y,\bar{J}X),\bar{J}N)-3\{{g}(\bar{h^{*}}(X,\bar{J}Y)-\bar{h^{*}}(Y,\bar{J}X),N)\}=0$.\\
	Similarly,\;$\bar{g}(\bar{\Omega}_X Y -\theta(Y)\{lX+m\bar{J}X\},Z) = \bar{g}(\bar{\Omega}_Y X -\theta(X)\{lY+m\bar{J}Y\},Z)$,\\
	Using equations (\ref{a}), (\ref{b}), we derive\\
	$\bar{g}(\bar{\Omega}_X \bar{J}Y,\bar{J}Z)-3 \bar{g}(\bar{\Omega}_X \bar{J}Y,Z) + 3l\theta(\bar{J}Y)\bar{g}(X,Z) -l\theta(\bar{J}Y)\bar{g}(\bar{J}X,Z)-m\theta(\bar{J}Y)\bar{g}(X,Z) = \bar{g}(\bar{\Omega}_Y \bar{J}X,\bar{J}Z)-3 \bar{g}(\bar{\Omega}_Y \bar{J}X,Z) +3l\theta(\bar{J}X)\bar{g}(Y,Z)-l\theta(\bar{J}X)\bar{g}(\bar{J}Y,Z)-m\theta(\bar{J}X)\bar{g}(Y,Z)$,\\
	On applying (\ref{g2}), (\ref{g6}), (\ref{e6}) and (\ref{e22}), we get desired result.\\
\end{proof}
\begin{theorem}
	The distribution $B^{'}$ of a screen generic lightlike submanifold  $M$ of a locally bronze semi-Riemannian manifold $\bar{M}$ with an $(l,m)$- type connection $\bar{\Omega}$ is integrable if and only if\\
	(i)\;$\Omega_Y^{*}fZ + \bar{A}_{wY}Z+ 3\Omega_Z^{*}Y - \Omega_Z^{*}fY - \bar{A}_{wZ}Y- 3\Omega_Y^{*}Z$ has no component in $\Gamma(B_o)$,\\
	(ii)\;$\bar{A}_ {wY}  Z + \bar{h^{*}}(Y,fZ)+3\bar{h^{*}}(Z,Y)= \bar{A}_{wZ} Y + \bar{h^{*}}(Z,fY)+3\bar{h^{*}}(Y,Z),$\\
	for any $Y,Z\in\Gamma(B^{'})$, $X\in\Gamma(B_o)$, $N\in\Gamma(ltr(TM))$.
\end{theorem}
\begin{proof}
	$B^{'}$ is integrable if and only if\; $\bar{g}([Y,Z],N)=0,\quad \bar{g}([Y,Z],X)=0,$\\ 
	\quad Since, $\bar{\Omega}$ is an $(l,m)$-type connection, therefore applying equations $(\ref{b})$, $(\ref{c})$, $(\ref{e5})$, $(\ref{e6})$, $(\ref{e7})$ and $(\ref{g6})$, we obtain\\
	$\bar{g}([Y,Z],N)= \bar{g}(\bar \Omega_Y Z-l\theta(Z)Y-m\theta(Z)\bar{J}Y,N)-\bar{g}(\bar \Omega_Z Y-l\theta(Y)Z-m\theta(Y)\bar{J}Z,N) = {g}(\Omega_Y fZ,\bar{J}N)+{g}(-\bar{A}_{wZ}Y,\bar{J}N)-3{g}(\Omega_Y Z,\bar{J}N)-{g}(\Omega_Z fY,\bar{J}N)-{g}(-\bar{A}_{wY}Z,\bar{J}N)+3{g}(\Omega_Z Y,\bar{J}N) = 0$ ,\\
	Now, using equation (\ref{e22}), we get\\
	${g}(\bar{h^{*}}(Y,fZ),\bar{J}N)-{g}(\bar{A}_ {wZ}  Y,\bar{J}N)-3{g}(\bar{h^{*}}(Y,Z),\bar{J}N)-{g}(\bar{h^{*}}(Z,fY),\bar{J}N)+{g}(\bar{A}_ {wY}  Z,\bar{J}N)+3{g}(\bar{h^{*}}(Z,Y),\bar{J}N)=0$.\\
	Similarly, \;$\bar{g}([Y,Z],X)= {g}(\Omega_Y^{*}fZ,\bar{J}X)+{g}(\bar{A}_{wY}Z,\bar{J}X)+3{g}(\Omega_Z^{*}Y,\bar{J}X)+(9m+3l)\{\theta(Z){g}(fY,X)-\theta(Y){g}(fZ,X)\}+3m\{\theta(Z){g}(Y,X)-\theta(Y){g}(Z,X)\}-{g}(\Omega_Z^{*}fY,\bar{J}X)-{g}(\bar{A}_{wZ}Y,\bar{J}X)-3{g}(\Omega_Y^{*}Z,\bar{J}X)-(l+3m)\{\theta(\bar{J}Z){g}(fY,X)-\theta(\bar{J}Y){g}(fZ,X)\}-m\{\theta(\bar{J}Z){g}(Y,X)-\theta(\bar{J}Y){g}(Z,X)\}=0,$\\
	Then, from the equations (\ref{g2}) and (\ref{g3}), we get the desired result.
\end{proof}
\begin{theorem}
	For screen generic lightlike submanifold $M$ of a locally bronze semi-Riemannian manifold with an $(l,m)$- type connection $\bar{\Omega}$, the distribution $B$ is integrable if and only if\\ 
	${g}(\Omega_X \bar{J}Y,fZ)+{g}(\bar{h^{s}}(X,\bar{J}Y),\bar{J}Z)+3 {g}(\Omega_Y \bar{J}X,Z)= {g}(\Omega_Y \bar{J}X,fZ)+{g}(\bar{h^{s}}(Y,\bar{J}X),\bar{J}Z)+3{g}(\Omega_X \bar{J}Y,Z),$\\
	for any $X,Y\in\Gamma(B)$ , $Z\in\Gamma(B^{'})$ and $N\in\Gamma(ltr(TM))$.
\end{theorem}
\begin{proof}
	Since $\bar{M}$ be a locally bronze semi-Riemannian manifold equipped with an $(l,m)$-type connection $\bar{\Omega}$.\\
	Therefore, $\bar{g}([X,Y],Z)=\bar{g}(\bar{\Omega}_X\bar{J}(\bar{J}Y),Z)-3\bar{g}(\bar{\Omega}_X \bar{J}Y,Z)-\bar{g}(\bar{\Omega}_Y\bar{J}(\bar{J}X),Z)+3\bar{g}(\bar{\Omega}_Y \bar{J}X,Z)-l\theta(Y)\bar{g}(X,Z)-m\theta(Y)\bar{g}(\bar{J}X,Z)+l\theta(X)\bar{g}(Y,Z)+m\theta(X)\bar{g}(\bar{J}Y,Z)$,\\
	Now, using (\ref{g6}) and (\ref{e6}), we obtain\\
	$\bar{g}([X,Y],Z)= {g}(\Omega_X \bar{J}Y,fZ)+{g}(\bar{h^{s}}(X,\bar{J}Y),wZ)+3l\theta(\bar{J}Y){g}(X,Z)-l\theta(\bar{J}Y){g}(\bar{J}X,Z)-m\theta(\bar{J}Y){g}(X,Z)-3{g}(\Omega_X \bar{J}Y,Z)-{g}(\Omega_Y \bar{J}X,fZ)-{g}(\bar{h^{s}}(Y,\bar{J}X),wZ)-3l\theta(\bar{J}X){g}(Y,Z)+l\theta(\bar{J}X) {g}(\bar{J}Y,Z)+m\theta(\bar{J}X) {g}(Y,Z)+3 {g}(\Omega_Y \bar{J}X,Z) $,\\
	$B$ is integrable, therefore $\bar{g}([X,Y],Z)=0$. Also, $M$ is screen generic lightlike submanifold.\\
	Thus, ${g}(\Omega_X \bar{J}Y,fZ)+ {g}(\bar{h^{s}}(X,\bar{J}Y),\bar{J}Z)+3 {g}(\Omega_Y \bar{J}X,Z)- {g}(\Omega_Y \bar{J}X,fZ)-{g}(\bar{h^{s}}(Y,\bar{J}X),\bar{J}Z)-3{g}(\Omega_X \bar{J}Y,Z)=0$, which completes the proof.\\
\end{proof}

\begin{theorem}
	Let $M$ be a screen generic lightlike submanifold of a locally bronze semi-Riemannian manifold $\bar{M}$ with an $(l,m)$- type connection $\bar{\Omega}$. Then, the distribution $B_o$ is parallel  if and only if\\
	(i)\;${g}(\bar{h^{*}}(X,\bar{J}Y),\bar{J}N)=3{g}(\bar{h^{*}}(X,\bar{J}Y),N),$\\
	(ii)\;${g}(\Omega_X^{*} \bar{J}Y,fZ)+{g}(\bar{h^{s}}(X,\bar{J}Y),\bar{J}Z)=3{g}(\Omega_X^{*}\bar{J}Y,Z),$\\
	
	for all $X,Y\in\Gamma(B_o)$, $Z\in\Gamma(B^{'})$ and $N\in\Gamma(ltr(TM))$.
\end{theorem}
\begin{proof}
	Since the distribution $B_o$ is parallel. Therefore, for all  $X,Y\in\Gamma(B_o)$, $Z\in\Gamma(B^{'})$ and $N\in\Gamma(ltr(TM))$,
	\begin{equation}\label{pp1}
		{g}(\Omega_X Y,Z)=0,\; {g}(\Omega_X Y,N)=0,
	\end{equation}
	Using the concept of locally bronze semi-Riemannian manifold, we get\\
	${g}(\Omega_X Y,N)=\bar{g}(\bar{\Omega}_X \bar{J}(\bar{J}Y),N)-3\bar{g}(\bar{\Omega}_X \bar{J}Y,N)$,\\
	Employing equations (\ref{b}) and (\ref{e5}), we obtain\\
	${g}(\Omega_X Y,N)= {g}(\Omega_X \bar{J}Y,\bar{J}N)+ 3l \theta(\bar{J}Y){g}(X,N)+l\theta(Y){g}(X,N)-l\theta(\bar{J}Y){g}(\bar{J}X,N)+m\theta(Y){g}(\bar{J}X,N)-m\theta(\bar{J}Y){g}(X,N)-3{g}(\Omega_X \bar{J}Y,N)$,\\
	Equations (\ref{e22}) and (\ref{pp1}) leads to \\
	${g}(\Omega_X Y,N)={g}(\Omega_X^{*}\bar{J}Y,\bar{J}N)+{g}(\bar{h^{*}}(X,\bar{J}Y),\bar{J}N)-3{g}(\Omega_X^{*}\bar{J}Y,\bar{J}N)-3{g}(\bar{h^{*}}(X,\bar{J}Y),\bar{J}N)=0.$\\
	Similarly, from equations (\ref{a}), (\ref{b}), (\ref{g6}) and (\ref{e5}), we get\\
	${g}(\Omega_X Y,Z)= \bar{g}(\bar{\Omega}_X \bar{J}Y,fZ)+\bar{g}(\bar{\Omega}_X \bar{J}Y,wZ)-3\bar{g}(\bar{\Omega}_X \bar{J}Y,Z)+3l \theta(\bar{J}Y)\bar{g}(X,Z)+l\theta(Y)\bar{g}(X,Z)-l\theta(\bar{J}Y)\bar{g}(\bar{J}X,Z)+m\theta(Y)\bar{g}(\bar{J}X,Z)-m\theta(\bar{J}Y)\bar{g}(X,Z)$,\\
	From (\ref{e6}), (\ref{e22}) and (\ref{pp1}), we attain \\
	${g}(\Omega_X Y,Z)={g}(\Omega_X^{*} \bar{J}Y,fZ)+{g}(\bar{h^{s}}(X,\bar{J}Y),\bar{J}Z)-3{g}(\Omega_X^{*}\bar{J}Y,Z)=0.$
\end{proof}
\begin{theorem}
	For a locally bronze semi-Riemannian manifold $\bar{M}$ with an $(l,m)$-type connection $\bar{\Omega}$, the distribution $B^{'}$ of a screen generic lightlike submanifold  $M$ is parallel if and only if\\
	(i)\;${g}(\bar{h^{*}}(Y,fZ),\bar{J}N)+3{g}(\bar{A}_{wZ} Y,N)=3{g}(\bar{h^{*}}(Y,fZ),N)+{g}(\bar{A}_{wZ} Y,\bar{J}N),$\\
	(ii)\;${g}(\Omega_Y^{*}fZ,\bar{J}X)+3{g}(\bar{A}_{wZ} Y,X)={g}(\bar{A}_{wZ} Y,\bar{J}X)+3{g}(\Omega_Y^{*}fZ,X),$\\
	for any $Y,Z\in\Gamma(B^{'})$, $X\in\Gamma(B_o)$ and $N\in\Gamma(ltr(TM))$.
\end{theorem}
\begin{proof}
	Let the distribution $B^{'}$ is parallel. Therefore, for all $Y,Z\in\Gamma(B^{'})$, $X\in\Gamma(B_o)$ and $N\in\Gamma(ltr(TM))$, \; $g(\Omega_Y Z,X)=0,\quad g(\Omega_Y Z,N)=0,$\\
	Now, using equations $(\ref{a})$ and $(\ref{e5})$, we obtain\\
	$g(\Omega_Y Z,N)=\bar{g}(\bar{J}(\bar{\Omega}_Y \bar{J}Z),N)+\bar{g}(l\{3\theta(\bar{J}Z)Y+\theta(Z)Y-\theta(\bar{J}Z)\bar{J}Y\},N)+\bar{g}(m\{\theta(Z)\bar{J}Y-\theta(\bar{J}Z)Y\},N)-3\bar{g}(\bar{\Omega}_Y \bar{J}Z,N)$,\\
	Employing (\ref{b}), (\ref{g6}), (\ref{e6}) and (\ref{e8}), we attain\\
	$g(\Omega_Y Z,N)={g}(\Omega_Y fZ,\bar{J}N)+{g}(\bar{h^{l}}(Y,fZ),\bar{J}N)+{g}(\bar{h^{s}}(Y,fZ),\bar{J}N)+{g}(-\bar{A}_{wZ}Y,\bar{J}N)+{g}(\bar{\Omega}_Y^{s}wZ,\bar{J}N)+{g}(\bar{D^{l}}(Y,wZ),\bar{J}N)-3{g}(\Omega_Y fZ,N)-3{g}(\bar{h^{l}}(Y,fZ),N)-3{g}(\bar{h^{s}}(Y,fZ),N)-3{g}(-\bar{A}_{wZ}Y,N)-3{g}(\bar{\Omega}_Y^{s}wZ,N)-3{g}(\bar{D^{l}}(Y,wZ),N),$\\
	Further, from equation(\ref{e22}), we get\\
	$g(\Omega_Y Z,N)={g}(\Omega_Y^{*}fZ,\bar{J}N)+{g}(\bar{h^{*}}(Y,fZ),\bar{J}N)+{g}(-\bar{A}_{wZ}Y,\bar{J}N)-3{g}(\Omega_Y^{*}fZ,N)-3{g}(\bar{h^{*}}(Y,fZ),N)-3{g}(-\bar{A}_{wZ}Y,N)=0. $\\
	Similarly, $g(\Omega_Y Z,X)={g}(\Omega_Y \bar{J}Z,\bar{J}X)+ 3l \theta(\bar{J}Z){g}(Y,X)+l\theta(Z){g}(Y,X)-l\theta(\bar{J}Z){g}(\bar{J}Y,X)+m\theta(Z){g}(\bar{J}Y,X)-m\theta(\bar{J}Z){g}(Y,X)-3{g}(\Omega_Y \bar{J}Z,X)$,\\
	Equations (\ref{e6}), (\ref{e8}) and (\ref{e22}) leads to\\
	$g(\Omega_Y Z,X)={g}(\Omega_Y^{*}fZ+\bar{h^{*}}(Y,fZ),\bar{J}X)+{g}(-\bar{A}_{wZ}Y,\bar{J}X)-3{g}(\Omega_Y^{*}fZ+\bar{h^{*}}(Y,fZ),X)-3{g}(-\bar{A}_{wZ}Y,X)=0 $.\\
\end{proof}

\begin{definition}\cite{SG2}
	$M$ is a screen $B$-geodesic screen generic lightlike submanifold if its second fundamental forms  satisfies 
	\begin{equation}\label{h1}
		h(X,Y)=0, \forall \quad X,Y\in\Gamma(B),
	\end{equation}
	\begin{remark}
		We say that $M$ is $B$-geodesic screen generic lightlike submanifold endowed with an $(l,m)$-type connection if its second fundamental forms satisfy the following conditions
		\begin{equation}\label{h2}
			\bar h^{l}(X,Y)=\bar h^{s}(X,Y)=0, \forall \quad X,Y\in\Gamma(B),
		\end{equation}
	\end{remark}	
\end{definition}
\begin{theorem}
	The distribution $B$ of a screen generic lightlike submanifold of a locally bronze semi-Riemannian manifold defines totally geodesic foliation in $\bar{M}$ with an $(l,m)$-type connection $\bar{\Omega}$ if and only if M is $B$-geodesic and $B$ is parallel with respect to to $\bar \Omega$ on $M$.
\end{theorem}
\begin{proof}
	Suppose that $B$ defines totally geodesic foliation in $\bar{M}$, then for all $X,Y\in\Gamma(B)$, \; $\bar{\Omega}_X Y\in\Gamma(B)$\; and \; $\bar{g}(\bar{\Omega}_X Y,\xi)=0,\quad \bar{g}(\bar{\Omega}_X Y,W)=0 \quad \bar{g}(\bar{\Omega}_X Y,Z)=0,$\;for all $\xi\in\Gamma(Rad(TM))$, $Z\in\Gamma(B^{'})$ and $W\in\Gamma(S(TM^{\perp})).$\\
	Using equations $(\ref{e6})$, $(\ref{e10})$ and $(\ref{e11})$, we get\\
	$\bar{g}(\bar{\Omega}_X Y , \xi)= {g}(\bar h^{l}(X,Y),\xi)= {g}( h^{l}(X,Y),\xi) + m\theta(Y){g}(w_l X,\xi)$,\\
	$\bar{g}(\bar{\Omega}_X Y,W)= {g}( \bar h^{s}(X,Y),W)= {g}( h^{s}(X,Y),W)+m\theta(Y){g}(w_s X,W)$,\\
	Since $M$ is screen generic lightlike submanifold of $\bar{M}$. Therefore, the distribution $B$ is invariant with respect to bronze structure $\bar{J}$. Hence, $w_l X$ and  $w_s X$ vanishes.\\
	Thus, $\bar h^{s}(X,Y)=\bar h^{l}(X,Y)=0$, which implies that $M$ is $B$-geodesic and $B$ is parallel with respect to $\Omega$ on $M$.\\
	Conversly, let $M$ be $B$-geodesic and parallel with respect to $\Omega$ on $M$. Therefore, $\bar h^{s}(X,Y)=\bar h^{l}(X,Y)=0$ $\forall$ $X,Y\in\Gamma(B)$, which implies that $\bar{\Omega}_X Y\in\Gamma(B)$ \\
\end{proof}
\begin{definition}\cite{SG2}
	$M$ is said to be mixed geodesic screen generic lightlike submanifold of a locally bronze semi-Riemannian manifold if its second fundamental form  satisfies 
	\begin{equation}\label{h4}
		h(X,Y)=0 ,\quad \forall \; X\in\Gamma(B),\; Y\in\Gamma(B^{'})
	\end{equation}
\end{definition}
\begin{remark}
	If its second fundamental forms satisfies the conditions
	\begin{equation}\label{h3}
		\bar h^{l}(X,Y)=\bar h^{s}(X,Y)=0, \forall \quad X\in\Gamma(B),\quad Y\in\Gamma(B^{'}),
	\end{equation}
	then $M$ is said to be mixed geodesic screen generic lightlike submanifold with an $(l,m)$-type connection.
\end{remark}
\begin{theorem}
	Let $M$ be a screen generic lightlike submanifold of a locally bronze semi-Riemannian manifold$(\bar{M},\bar{g},\bar{J})$ equipped with an $(l,m)$-type connection $\bar{\Omega}$. Then $M$ is mixed geodesic if and only if the following statements holds\\
	(i)\;${g}(\bar h^{l}(X,fZ),\xi)+l\theta(Z){g}(\bar{J}X,\xi)+3m\theta(Z){g}(\bar{J}X,\xi)+m\theta(Z){g}(X,\xi)=
	-{g}(\bar D^{l}(X,wZ),\xi)+l\theta(\bar{J}Z){g}(X,\xi)+m\theta(\bar{J}Z){g}(\bar{J}X,\xi),$\\
	(ii)\;$3{g}(\bar h^{s}(X,fZ)+\bar{\Omega}_X ^{s}wZ,W)+{g}(\bar{A}_{wZ} X, fW)={g}(\Omega_X fZ+\bar h^{s}(X,fZ),\bar{J}W)+{g}(\bar{\Omega}_X^{s} wZ, wW),$\\
	for all $X\in\Gamma(B)$, $Z\in\Gamma(B^{'})$ and $W\in\Gamma(S(TM^{\perp}))$.
\end{theorem}
\begin{proof}
	Since $M$ is mixed geodesic, therefore\; ${g}(\bar h^{l}(X,Z),\xi)=0,\; {g}(\bar h^{s}(X,Z),W)=0$ \\
	for all $ X\in\Gamma(B)$, $Z\in\Gamma(B^{'})$, $\xi\in\Gamma(Rad(TM))$ and $W\in\Gamma(S(TM^{\perp}))$.\\
	Using Gauss formula, we obtain \\
	$\bar{g}(\bar{\Omega}_X Z, \xi)=0,\; \bar{g}(\bar{\Omega}_X Z, W)=0$\\
	From equations (\ref{g1}), (\ref{g6}) and $(\ref{e5})$, we have\\
	$\bar{g}(\bar{\Omega}_X Z, \xi)=\bar{g}(\bar{\Omega}_X Z, \bar{J}\xi)= \bar{g}(\bar{\Omega}_X fZ,\xi)+ \bar{g}(\bar{\Omega}_X wZ,\xi)+l\theta(Z)\bar{g}(\bar{J}X,\xi)+3m\theta(Z)\bar{g}(\bar{J}X,\xi)+m\theta(Z)\bar{g}(X,\xi)-l\theta(\bar{J}Z)\bar{g}(X,\xi)-m\theta(\bar{J}Z)\bar{g}(\bar{J}X,\xi)=0$,\\
	Similarly, using equations (\ref{a}), (\ref{b}), (\ref{g6}) and $(\ref{e5})$ , we get\\
	$\bar{g}(\bar{\Omega}_X Z, W)= \bar{g}(\bar{\Omega}_X \bar{J}(\bar{J}Z),W)-3\bar{g}(\bar{\Omega}_X \bar{J}{Z},W)=\bar{g}(\bar{\Omega}_X fZ,fW)+ \bar{g}(\bar{\Omega}_X fZ,wW)+\bar{g}(\bar{\Omega}_X wZ,fW)+\bar{g}(\bar{\Omega}_X wZ,wW)-3 \bar{g}(\bar{\Omega}_X fZ,W)- 3\bar{g}(\bar{\Omega}_X wZ,W)$,\\
	Employing equations  $(\ref{e6})$ and $(\ref{e8})$, we attain\\
	$\bar{g}(\bar{\Omega}_X Z, W)={g}({\Omega}_X fZ,fW)+ {g}(\bar h^{s}(X,fZ),wW)+ {g}(-\bar{A}_{wZ}X,fW)+ {g}(\bar{\Omega}_X^{s}wZ,wW)-3{g}(\bar h^{s}(X,fZ),W)-3{g}(\bar{\Omega}_X^{s}wZ,W) =0.$
\end{proof}
\begin{theorem}
	Let $M$ be a screen generic lightlike submanifold of a locally bronze semi-Riemannian manifold $\bar{M}$ equipped with an $(l,m)$-type connection $\bar{\Omega}$. Then, for any $X\in\Gamma(B_o)$, $Z\in\Gamma(B^{'})$, we have \\
	$f(\Omega_X fZ-\bar{A}_{wZ}X)+B(\bar h^{s}(X,fZ)+\bar{\Omega}_X^{s}wZ)+3\bar{A}_{wZ}X+3l\theta(\bar{J}Z)X+l\theta(Z)X+m\theta(Z)\bar{J}X=3 \Omega_X fZ +l\theta(\bar{J}Z)\bar{J}X+m\theta(\bar{J}Z)X+\Omega_X Z,$
\end{theorem}
\begin{proof}
	Using (\ref{a}), (\ref{g6}) and (\ref{e5}), we get\\
	$\bar{\Omega}_X Z =\bar{J}(\bar{\Omega}_X fZ)+\bar{J}(\bar{\Omega}_X wZ)+3l \theta(\bar{J}Z)X+l\theta(Z)X-l\theta(\bar{J}Z)\bar{J}X+m\theta(Z)\bar{J}X-m\theta(\bar{J}Z)X-3\bar{\Omega}_X fZ- 3 \bar{\Omega}_X wZ$,\\
	From (\ref{s1}), (\ref{s2}), (\ref{e6}) and (\ref{e8}), we obtain \\
	$\bar{\Omega}_X Z=f(\Omega_X fZ) + w(\Omega_X fZ) + B(\bar h^{l}(X,fZ))+C(\bar h^{l}(X,fZ))+B(\bar h^{s}(X,fZ))+C(\bar h^{s}(X,fZ))+f(-\bar{A}_{wZ}X)+ w(-\bar{A}_{wZ}X)+ B(\bar{\Omega}_X^{s}wZ)+ C(\bar{\Omega}_X^{s}wZ)+ B(\bar D^{l}(X,wZ))+C(\bar D^{l}(X,wZ))+3l \theta(\bar{J}Z)X+l\theta(Z)X-l\theta(\bar{J}Z)\bar{J}X+m\theta(Z)\bar{J}X-m\theta(\bar{J}Z)X-3{\Omega}_X fZ- 3 \bar h^{l}(X,fZ)- 3 \bar h^{s}(X,fZ)+ 3 \bar{A}_{wZ}X- 3\bar{\Omega}_X^{s}wZ-3\bar D^{l}(X,wZ) $,\\ 
	Comparing tangential parts of above equation, we have \\
	$\Omega_X Z= f(\Omega_X fZ-\bar{A}_{wZ}X)+B(\bar h^{s}(X,fZ)+\bar{\Omega}_X^{s}wZ)+3\bar{A}_{wZ}X+3l\theta(\bar{J}Z)X+l\theta(Z)X+m\theta(Z)\bar{J}X-3 \Omega_X fZ -l\theta(\bar{J}Z)\bar{J}X-m\theta(\bar{J}Z)X.$
\end{proof}

\begin{theorem}
	For a locally bronze semi-Riemannian manifold $(\bar{M},\bar{g},\bar{J})$ with an $(l,m)$-type connection $\bar{\Omega}$, a screen generic lightlike submanifold  $M$ is mixed geodesic if and only if the following conditions hold\\
	(i)\; $C(\bar h^{l}(X,fZ) + \bar D^{l}(X,wZ))=0,$\\
	(ii)\; $w(\Omega_X fZ-\bar{A}_{wZ}X)= - C(\bar h^{s}(X,fZ)+\bar{\Omega}_X ^{s} wZ),$\\
	(iii)\; $\bar h^{s}(X,fZ)+ \bar D^{l}(X,wZ)= - \bar h^{l}(X,fZ)-\bar{\Omega}_X ^{s} wZ,$\\
	for any $X\in\Gamma(B)$,\; $Z\in\Gamma(B^{'})$.
\end{theorem}
\begin{proof}
	Let $M$ is mixed geodesic screen generic lightlike submanifold, then\\
	From $(\ref{eq21})$, $(\ref{a})$, $(\ref{g6})$ and $(\ref{e1})$, we get\\
	$h(X,Z)=\bar{J}(\bar{\Omega}_X fZ)+ \bar{J}(\bar{\Omega}_X wZ)+3l \theta(\bar{J}Z)X+l\theta(Z)X-l\theta(\bar{J}Z)\bar{J}X+m\theta(Z)\bar{J}X-m\theta(\bar{J}Z)X-3\bar{\Omega}_X fZ-3 \bar{\Omega}_X wZ-\Omega_X Z ,$\\
	Employing $(\ref{s1})$, $(\ref{s2})$, $(\ref{e6})$ and $(\ref{e8})$ on above equation and compare the transversal parts, we obtain \\
	$h(X,Z)= w(\Omega_X fZ)+ C(\bar h^{l}(X,fZ))+C(\bar h^{s}(X,fZ))+ w(-\bar{A}_{wZ}X)+C(\bar{\Omega}_X ^{s}wZ)+C(\bar D^{l}(X,fZ))-3\bar h^{l}(X,fZ)-3 \bar h^{s}(X,fZ) - 3 \bar{\Omega}_X^{s}wZ-3\bar D^{l}(X,wZ)=0. $
\end{proof}

\section{Totally umbilical screen generic lightlike submanifolds}

\begin{definition}\cite{TU1}
	A lightlike submanifold $(M,g)$ of a semi-Riemannian manifold $(\bar{M},\bar{g})$ is totally umbilical in $\bar{M}$ if there is a smooth transversal vector field $H\in\Gamma(ltr(TM))$ on $M$, called the transversal curvature vector field of $M$, such that, for all $ X,Y\in\Gamma(TM)$,
	\begin{equation}\label{u1}
		h(X,Y)= Hg(X,Y),
	\end{equation} 
\end{definition}
Using Gauss and Weingarten equations of $\bar{M}$, we observe that $M$ is totally umbilical if and only if on each coordinate neighborhood $U$, there exist smooth vector fields $H^{l}\in\Gamma(ltr(TM))$ and $H^{s}\in\Gamma(S(TM^{\perp}))$, such that
\begin{equation}\label{u2}
	h^{l}(X,Y)= H^{l}g(X,Y),\;h^{s}(X,Y)= H^{s}g(X,Y),\; D^{l}(X,W)=0,
\end{equation}
for all $ X,Y\in\Gamma(TM), W\in\Gamma(S(TM^{\perp}))$.
\begin{lemma}\label{L1}
	let $M$ be a totally umbilical proper screen generic lightlike submanifold of a locally bronze semi-Riemannian manifold $\bar{M}$ with an $(l,m)$-type connection $\bar{\Omega}$. Then,\\
	(i)${\Omega}_X fY-\bar{A}_{w_l Y} X - \bar{A}_{w_s Y} X- l\theta(\bar{J}Y)X- m\theta(\bar{J}Y) fX =f \Omega_X Y + B \bar{h^{s}}(X,Y)- l \theta(Y)fX-3m\theta(Y)fX-m\theta(Y)X$,\\
	(ii)$ \bar h^{l}(X,fY) + \bar{\Omega}_X ^{l}w_l Y -m\theta(\bar{J}Y)w_l X= w_l \Omega_X Y + C \bar{h^{l}}(X,Y)-l\theta(Y)w_l X -3 m\theta(Y)w_l X$,\\
	(iii)$\bar{h^{s}}(X,fY)+ \bar{D^{s}}(X,w_l Y) + \bar{\Omega}_X^{s} w_s Y -m\theta(\bar{J}Y)w_s X =  w_s \Omega_X Y + C \bar{h^{s}}(X,Y)-l\theta(Y)w_s X -3 m\theta(Y)w_s X,$ for any \; $X,Y\in\Gamma(TM)$
	\begin{proof}
		Since $\bar{M}$ is a locally bronze semi-Riemannian manifold, it implies that $\bar{\nabla}_X \bar{J}Y= \bar{J}(\bar{\nabla}_X Y)$,\\
		Following equation $(\ref{e1})$, we obtain\\
		$\bar{\Omega}_X \bar{J}Y-l\theta(\bar{J}Y)X-m\theta(\bar{J}Y)\bar{J}X = \bar{J}(\bar{\Omega}_X Y)- l \theta(Y)\bar{J}X - m\theta(Y)\bar{J}^{2}X$,\\
		Using equations $(\ref{a})$, $(\ref{s1})$ and $(\ref{s2})$, we get\\
		$\bar{\Omega}_X fY + \bar{\Omega}_X w_l Y + \bar{\Omega}_X w_s Y - l \theta(\bar{J}Y)X - m\theta(\bar{J}Y)fX- m\theta(\bar{J}Y)w_l X- m\theta(\bar{J}Y)w_s X = \bar{J}(\bar{\Omega}_X Y)-l\theta(Y)fX-l\theta(Y)w_l X-l\theta(Y)w_s X- 3m\theta(Y)fX-3m\theta(Y)w_l X-3m\theta(Y)w_s X-m\theta(Y)X$,\\
		From equations $(\ref{e6})$, $(\ref{e7})$ and $(\ref{e8})$ and $(\ref{u2})$, we attain\\
		$\Omega_X fY + \bar{h^{l}}(X,fY) + \bar{h^{s}}(X,fY) - \bar{A}_{w_l Y} X + \bar{\Omega}_{X}^{l} {w_l Y} + \bar{D}^{s}(X, w_l Y) - \bar{A}_{w_s Y} X +\bar{\Omega}_{X}^{l} {w_s Y}-  l \theta(\bar{J}Y)X -m\theta(\bar{J}Y)fX- m\theta(\bar{J}Y)w_l X- m\theta(\bar{J}Y)w_s X = f \Omega_X Y +{w_l} \Omega_{X} Y+ {w_s} \Omega_{X} Y + C\bar{h^{l}}(X,Y)+ B\bar{h^{s}}(X,Y+ C\bar{h^{s}}(X,Y)-l\theta(Y)fX-l\theta(Y)w_l X-l\theta(Y)w_s X- 3m\theta(Y)fX-3m\theta(Y)w_l X- 3m\theta(Y)w_s X -m\theta(Y)X,$\\
		Thus the required result is obtained by comparing the tangential and normal parts of above equation.
	\end{proof}
\end{lemma}
\begin{theorem}
	If $M$ is a totally umbilical proper screen generic lightlike submanifold  of a locally bronze semi-Riemannian manifold $\bar{M}$ with an $(l,m)$-type connection $\bar{\Omega}$, then $H^{s}\notin\Gamma(\mu)$.
\end{theorem}
\begin{proof}
	Let $M$ be a totally umbilical proper screen generic lightlike submanifold of a locally bronze semi-Riemannian manifold, then $\bar{\nabla}_X \bar{J}Y= \bar{J}(\bar{\nabla}_X Y)$. \\
	Now, using equations $(\ref{a})$, $(\ref{s1})$, $(\ref{s2})$, $(\ref{e1})$, $(\ref{e6})$ and comparing transversal parts, we obtain\\
	${h^{s}}(X,\bar{J}X)+ m\theta(\bar{J}X)w_s X = C{h^{s}}(X,X)+ Cm\theta(X)w_s X$, $\forall$ $X,Y\in\Gamma(B)$,\\
	From the concept of totally umbilical submanifold $M$ of $\bar{M}$, we get\\
	$g(X,\bar{J}X)H^{s}+ m\theta(\bar{J}X)w_s X = g(X,X)C H^{s}+ Cm\theta(X)w_s X$,\\
	Since $M$ is screen generic lightlike submanifold of $\bar{M}$. Therefore, the distribution $B$ is invariant with respect to $\bar{J}$. Also, if we take $X=Y$, then ,\quad $g(X,X)C H^{s}=0,\quad w_s X=0$.\\
	Hence, $H^{s}\notin\Gamma(\mu)$ and the proof is completed.
\end{proof}
\begin{theorem}
	Let $M$ be a totally umbilical proper screen generic lightlike submanifold of a locally bronze semi-Riemannian manifold $\bar{M}$ with an $(l,m)$-type connection $\bar{\Omega}$. Then the necessary and sufficient condition for induced connection to be metric is 
	\begin{equation*}
		l\theta(Y)g(X,Z) + m\theta(Y)g(fX,Z) = l\theta(Z)g(Y,X)+ m\theta(Z)g(Y,fZ),
	\end{equation*}
	$\forall \quad X,Y,Z\in\Gamma(B_o)$. 
\end{theorem}
\begin{proof}
	From lemma $(\ref{L1})$, we obtain\\
	$\bar h^{l}(X,fY) + \bar{\Omega}_X ^{l}w_l Y -m\theta(\bar{J}Y)w_l X= w_l \Omega_X Y + C \bar{h^{l}}(X,Y)-l\theta(Y)w_l X -3 m\theta(Y)w_l X,$\\
	Since $M$ is a totally umbillical proper screen generic lightlike submanifold, therefore\\
	$h^{l}(X,\bar{J}Y) + m\theta(\bar{J}Y)w_l X= C{h^{l}}(X,Y)+C m\theta(Y) w_l X$,\\
	 the distribution $B_o$ is invariant with respect to bronze structure $\bar{J}$. Thus, $w_l X=0$. Also, from equations $(\ref{u1})$ and $(\ref{u2})$, we attain\\
	$g(X,\bar{J}Y)H^{l} = g(X,Y)C H^{l}$,\\
	From this , $ 2g(X,\bar{J}Y)H^{l}=0$ is obtained. If we take $X=\bar{J}Y$, then we have $H^{l}=0$, that is equation $(\ref{u2})$ implies that  $h^{l}=0$,\\ 
	Employing equation $(\ref{e18})$, we have\\
	$(\Omega_X g)(Y,Z)=-l(\theta(Y)g(X,Z)+l\theta(Z)g(Y,X))-m(\theta(Y)g(fX,Z)+m\theta(Z)g(Y,fZ))=0,$\\
\end{proof}

\section{Minimal screen generic lightlike submanifolds}

A general notion of minimal lightlike submanifold $M$ of a semi-Riemannian manifold $\bar{M}$  introduced by Bejan-Duggal in \cite{E1} is as follows:
\begin{definition}
	A lightlike submanifold $(M,g,S(TM),S(TM^{\perp}))$, of a semi-Riemannian manifold $(\bar{M},\bar{g})$, is said to be minimal if :\\
	(i)\; $h^{s}=0$ on $Rad(TM)$ and\\
	(ii)\;trace\;$h=0$, where trace is written with respect to $g$ restricted to $S(TM)$.
\end{definition} 
It has been shown in \cite{E1} that the above definition is independent of $S(TM)$ and $S(TM^{\perp})$, but is dependent on $tr(TM)$.

\begin{xca}
	let $\bar{M}=\mathbb{R}_1^{11}$ be a semi-Riemannian space of signature (-,+,+,+,+,+,+,+,+,+,) with respect to the canonical basis $\{\partial y_1,\partial y_2,\partial y_3,\partial y_4,\partial y_5,\partial y_6,\partial y_7,\partial y_8,\partial y_9,\partial y_{10},\partial y_{11} \}$ and $M$ a submanifold of $(\mathbb{R}_1^{11},\bar{J})$ given by 
	\begin{equation*}
		y_1=0,\quad y_2=\frac{1}{2}(\sqrt{3}t_4+ t_2), 
	\end{equation*}
	\begin{equation*}
		y_3=\frac{1}{2}(-\sqrt{3}t_2+ t_4),\quad y_4=t_1,\quad y_5=t_4,
	\end{equation*}
	\begin{equation*}
		y_6 = sin t_5\; sinht_6,\quad y_7 =0,\quad y_8 = sint_5\; cosht_6,
	\end{equation*}
	\begin{equation*}
		y_9 = 0,\quad y_{10} = \sqrt{2}\;cost_5\; cosht_6,\quad y_{11} = 0.
	\end{equation*}
	
	Consider the bronze structure $\bar{J}$ defined by 
	$\bar{J}(y_1,y_2,y_3,y_4,y_5,y_6,y_7,y_8,y_9,y_{10},y_{11})=(\omega y_1,(3-\omega) y_2,(3-\omega) y_3,\omega y_4,(3-\omega) y_5,3 y_6+ y_7, y_6,3y_8 + y_9, y_8, 3y_{10}+y_{11},y_{10}).$\\
	
	Then, $TM$ is spanned by $\{B_1,B_2,B_3,B_4,B_5\}$, where
	\begin{equation*}
		B_1= \frac{1}{2}(-\sqrt{3}\partial y_3+\partial y _2),\quad B_2=\partial y_4,
	\end{equation*}
	\begin{equation*}
		B_3=\partial y_5+ \frac{1}{2}(\sqrt{3}\partial y_2+\partial y _3),
	\end{equation*}
	\begin{equation*}
		B_4= cost_5\; sinht_6\;\partial y_6 + cost_5 \;cosht_6 \;\partial y_8-\sqrt{2}\;sint_5\; cosht_6\;\partial y_{10},
	\end{equation*}
	\begin{equation*}
		B_5 = sint_5\; cosht_6\;\partial y_6 + sint_5 \;sinht_6\; \partial y_8 + \sqrt{2}\; cost_5 sinht_6\;\partial y_{10}.
	\end{equation*}
	where $Rad(TM)=Span\{B_3\}$ and $B_o=Span\{B_1,B_2\}$. It is proved that $\bar{J}B_3=(3-\omega)B_3$, which shows that $Rad(TM)$ is invariant under $\bar{J}$. Since, $\bar{J}B_1=(3-\omega)B_1$ , $\bar{J}B_2=\omega B_1$. Therefore, $B_{o}=Span\{B_1,B_2\}$, is invariant under $\bar{J}$.\\
	By direct calculations, we derive that lightlike transversal bundle $ltr(TM)$ is spanned by 
	\begin{equation*}
		N=\frac{1}{2}\{-\partial y_5+ \frac{1}{2}\partial y_3+\frac{\sqrt{3}}{2} \partial y_2\},
	\end{equation*}
	such that $\bar{J}N=(3-\omega)N$, which shows that $ltr(TM)$ is invariant under $\bar{J}$.\\
	Also, the screen transversal bundle $S(TM^{\perp})$ is spanned by
	\begin{equation*}
		W_1=\partial y_1,\quad W_2=cost_5\; sinht_6\;\partial y_7 + cost_5\; cosht_6\; \partial y_9-\sqrt{2}\;sint_5\; cosht_6\;\partial y_{11},
	\end{equation*} 
	\begin{equation*}
		W_3=sint_5\; cosht_6\;\partial y_7 + sint_5\; sinht_6\; \partial y_9 + \sqrt{2}\; cost_5\; sinht_6\;\partial y_{11},
	\end{equation*}
	\begin{equation*}
		W_4=-\sqrt{2}\;sinht_6\; cosht_6\;\partial y_7 + \sqrt{2}\;(sin^{2}t_5 + sinh^{2}t_6)\;\partial y_9 + sint_5\;cost_5\; \partial y_{11},
	\end{equation*}
	\begin{equation*}
		W_5=-\sqrt{2}\;sinht_6 \;cosht_6\;\partial y_6 + \sqrt{2}\;(sin^{2}t_5 + sinh^{2}t_6)\;\partial y_8 + sint_5\; cost_5\; \partial y_{10},
	\end{equation*}
	Since, $\bar{J}W_4= W_5$, $\bar{J}W_5= 3 W_5+ W_4$, it is easy to see that $\mu= Span\{W_4,W_5\}$ is invariant under $\bar{J}$ . Now, $B^{'}=Span\{B_4,B_5\}$ such that $\bar{J}B_4=3B_4+ W_2$ and $\bar{J}B_5=3B_5+ W_3$, which shows that $\bar{J}(B^{'})\not\subset S(TM^)$ and  $\bar{J}(B^{'})\not\subset S(TM^{\perp})$.\\
	Therefore, $M$ is a proper screen generic $1$-lightlike submanifold of $\mathbb{R}_1^{11}$.\\
	On the other hand, by direct calculation, we obtain
	\begin{equation*}
		\bar \nabla_{B_i} {B_j} = 0, \quad i= 1,2,3,\quad 1\leq j \leq 5,\quad and \quad h^{l}(B_4,B_4)=0,\quad h^{l}(B_5,B_5)=0,
	\end{equation*}
	\begin{equation*}
		h^{s}(B_4,B_4)=-\frac{\sqrt{2}\;sint_5\; cosht_6}{(sin^{2}t_5 + 2 sinh^{2}t_6)(1+sin^{2}t_5+2 sinh^{2}t_6)} W_5,
	\end{equation*}
	\begin{equation*}
		h^{s}(B_5,B_5)=\frac{\sqrt{2}\;sint_5 \;cosht_6}{(sin^{2}t_5 + 2 sinh^{2}t_6)(1+sin^{2}t_5+2 sinh^{2}t_6)} W_5,
	\end{equation*}
	Hence, for all $X\in\Gamma(TM)$, $h^{s}(X,B_3)=0$, \\
	That is $h^{s}=0$ on $Rad(TM)$ and $traceh|_{S(TM)} = h^{s}(B_4,B_4)+ h^{s}(B_5,B_5)=0$.\\
	Therefore, $M$ is minimal proper screen generic lightlike submanifold of $\mathbb{R}_1^{11}$.
\end{xca}

\begin{theorem}
	Let $M$ be a screen generic lightlike submanifold of a locally bronze semi-Riemannian manifold $(\bar{M},\bar{g},\bar{J})$ with an $(l,m)$-type connection $\bar{\Omega}$. Then the distribution $B_{o}$ is minimal if and only if\\
	$3\{g(fX,\bar{A}_W fX) + l\theta(W)g(X,f^{2}X) + m\theta(W)g(fX, f^{2}X)\} = -2 \{g(\bar{A}_W {fX},X)+ l\theta(W)g(X,fX) + m\theta(W)g(X,f^{2}X)\}$,\; for any $X\in\Gamma(B_{o})$ and $W\in\Gamma(S(TM^{\perp}))$.
\end{theorem}
\begin{proof} 
	$B_{o}$ is minimal if and only if
	\begin{equation}\label{mm1}
		\Omega_X X + {\Omega}_{\bar{J}X} {\bar{J}X}- l\{\theta(X)X + \theta(\bar{J}X)\Bar{J}X\} -m \{\theta(X)fX + \theta(\bar{J}X)f^{2}X\} \in\Gamma(B_{o}),\; for all\; X\in\Gamma(B_{o})
	\end{equation}
	Since $M$ is a screen generic lightlike submanifold of $\bar{M}$. Therefore, the distribution $B_o$ is invariant w.r.t $\bar{J}$. Also using equations $(\ref{e9})$, $(\ref{1})$, $(\ref{b1})$, $(\ref{b})$, $(\ref{3})$ and $(\ref{e15})$, we get \\
	$g({\Omega}_X X, \bar{J}W)-l\theta(X)g(X,\bar{J}W)-m\theta(X)g(fX,\bar{J}W)= g(\bar{J}X,\bar{A}_W X) + l\theta(W)g(X,\bar{J}X) + m\theta(W)g(\bar{J}X,fX) $,\\
	and \quad $g({\Omega}_{\bar{J}X} \bar{J}X, \bar{J}W)-l\theta(\bar{J}X)g(\bar{J}X,\bar{J}W)-m\theta(\bar{J}X)g(f^{2}X,\bar{J}W)= g(3 \bar{J}X + X,\bar{A}_W \bar{J}X) + l\theta(W)g(3\bar{J}X + X,\bar{J}X) + m\theta(W)g(3\bar{J}X + X,f^{2}X)$,\\
	for any $X\in\Gamma(B_{o})$ and $W\in\Gamma(S(TM^{\perp}))$.\\
	On the other hand , since the shape operator is symmetric on $S(TM)$, we obtain\\
	$g(\Omega_X X + {\Omega}_{\bar{J}X} {\bar{J}X}- l\{\theta(X)X + \theta(\bar{J}X)\Bar{J}X\} -m \{\theta(X)fX + \theta(\bar{J}X)f^{2}X\},\bar{J}W) = g(\bar{J}X,\bar{A}_W X + l\theta(W)X + m\theta(W) fX) +g(3\bar{J}X + X,\bar{A}_W \bar{J}X + l\theta(W)\bar{J}X + m\theta(W) f^{2}X)$,\quad for any $X\in\Gamma(B_{o})$ and $W\in\Gamma(S(TM^{\perp}))$ \\
	Employing $(\ref{mm1})$ on above equation, we attain\\
	$3\{g(fX,\bar{A}_W fX) + l\theta(W)g(X,f^{2}X) + m\theta(W)g(fX, f^{2}X)\} = -2 \{g(\bar{A}_W {fX},X)+ l\theta(W)g(X,fX) + m\theta(W)g(X,f^{2}X)\}$
	
\end{proof}

\begin{theorem}
	Let $M$ be a screen generic lightlike submanifold of a locally bronze semi-Riemannian manifold $\bar{M}$ with an $(l,m)$-type connection $\bar{\Omega}$. Then $M$ is minimal if and only if
	\begin{enumerate}
		\item $trace \bar A^{*}_{\xi_k}|_{S(TM)} = trace \bar A_{W_j} |_{S(TM)} = 0,$
		\item $\bar{g}(\bar D^{l}(X,W),Y)=l\theta(W)\bar{g}(X,Y)+m\theta(W)\bar{g}(fX,Y),$
	\end{enumerate}
	$\forall  X,Y\in\Gamma(Rad(TM)), W\in\Gamma(S(TM^{\perp}))$, where $dim(TM)=m$, $dim(tr(TM))= n$, $dim(Rad(TM))=r$ , \{$\xi_k$\} is basis of $Rad(TM)$ and $W_j \in\Gamma(S(TM^{\perp}))$ for $k\in\{1,2,...,r\}$ and $j\in\{1,2,.......,n-r\}$.
\end{theorem}
\begin{proof}
	$M$ being minimal submanifold implies that \\
	$trace \bar h |_{S(TM)} = trace \bar h |_{B_o} + trace \bar h |_{B^{'}} = \Sigma_{i=1}^{a} \bar h(B_i, B_i) + \Sigma_{j=1}^{b} \bar h(E_j, E_j) = 0$,\\
	and  $\bar h^{s} |_{Rad(TM)} =0$.\\
	If we choose an orthonormal basis of $S(TM)$ as $\{e_i\}_{i=1}^{m-r}$, then we get\\
	$trace \bar h |_{S(TM)}=\Sigma_{i=1}^{a}\epsilon _i (\bar h^{l}(e_i,e_i)+\bar h^{s}(e_i,e_i)) + \Sigma_{j=1}^{b}\epsilon _j (\bar h^{l}(e_j,e_j)+\bar h^{s}(e_j,e_j))$,\\
	\quad $= \Sigma_{i=1}^{a}\epsilon _i [\frac{1}{r}\Sigma_{k=1}^{r}\bar{g}(\bar h^{l}(e_i,e_i),\xi_k)N_k + \frac{1}{n-r}\Sigma_{j=1}^{n-r}\bar{g}(\bar h^{s}(e_i,e_i),W_j)W_j]+\\$
	\quad $ \Sigma_{j=1}^{b}\epsilon _j [\frac{1}{r}\Sigma_{k=1}^{r}\bar{g}(\bar h^{l}(e_j,e_j),\xi_k)N_k + \frac{1}{n-r}\Sigma_{j=1}^{n-r}\bar{g}(\bar h^{s}(e_j,e_j),W_j)W_j]$,\\
	where $\{W_1,W_2,.....,W_{n-r}\}$is an orthonormal basis of $S(TM^{\perp})$.\\
	Since, $\bar{g}(\bar h^{l}(e_i,e_i),\xi_k)N_k = g (\bar A^{*}_{\xi_{k}}e_i,e_i)N_k $ and
	$\bar{g}(\bar h^{s}(e_i,e_i),W_j)W_j = g(\bar A_{W_j} e_i,e_i)W_j$, we derive\\
	$trace \bar h |_{S(TM)}=trace \bar A^{*}_{\xi_k}|_{B_o\oplus B^{'}} + trace \bar A_{W_j} |_{B_o\oplus B^{'}}$,\\
	Therefore, $trace\bar A^{*}_{\xi_k}|_{B_o\oplus B^{'}} $ and $trace \bar A_{W_j} |_{B_o\oplus B^{'}}=0$. \\
	On the other hand, using concept of minimal screen generic lightlike submanifold with equation $(\ref{e20})$, we obtain\\
	$\bar{g}(\bar h^{s}(X,Y),W)=-\bar{g}(\bar D^{l}(X,W),Y)+l\theta(W)\bar{g}(X,Y)+m\theta(W)\bar{g}(fX,Y)$,\\
	$\forall\quad X,Y\in\Gamma(Rad(TM)), W\in\Gamma(S(TM^{\perp}))$, which proof is completed.
\end{proof}

\bibliographystyle{amsplain}

\end{document}